\documentclass[reqno,oneside]{amsart}

\usepackage[english]{babel} 
\usepackage[latin1]{inputenc} 
\usepackage[T1]{fontenc} 
\usepackage{times}
\usepackage{amsmath,amsthm,amssymb,amsfonts,mathrsfs,latexsym} 
\usepackage{verbatim} 
\usepackage[colorlinks, linkcolor=blue, citecolor=blue, urlcolor=blue, hypertexnames=true]{hyperref}  

\newtheorem{thm}{Theorem}

\theoremstyle{definition}
\newtheorem{rem}[thm]{Remark}

\renewcommand{\phi}{\varphi}
\renewcommand{\tilde}{\widetilde}
\renewcommand\bar\overline

\author{S{\o}ren Knudby}
\address{S\o ren Knudby
\newline Mathematical Institute, University of M\"unster,
\newline Einsteinstra\ss{}e 62, 48149 M\"unster, Germany.}
\email{knudby@uni-muenster.de}
\thanks{Supported by the Deutsche Forschungsgemeinschaft through the Collaborative Research Centre (SFB 878).}

\title[On connected Lie groups and the Approximation Property]{\texorpdfstring{On connected Lie groups and\\ the Approximation Property}{Connected Lie groups and the Approximation Property}}
\date{\today}

\begin{document}

\begin{abstract}
Recently, a complete characterization of connected Lie groups with the Approximation Property was given. The proof used the newly introduced property~(T$^*$). We present here a short proof of the same result avoiding the use of property (T$^*$). Using property (T$^*$), however, the characterization is extended to every almost connected group. We end with some remarks about the impossibility of going beyond the almost connected case.
\end{abstract}

\maketitle
\thispagestyle{empty}

The Fourier algebra $A(G)$ of a locally compact group $G$ was introduced by Eymard \cite{MR0228628} and can be defined as the coeffient space of the left regular representation of $G$ on $L^2(G)$,
$$
A(G) = \{ f*\widetilde g \mid f,g\in L^2(G)\},
$$
where $\tilde g(x) = \overline{g(x^{-1})}$. One can identify $A(G)$ with the predual of the group von Neumann algebra associated with the left regular representation. A multiplier $\phi$ of $A(G)$ is called completely bounded if its transposed operator $M_\phi$ is a completely bounded operator on the group von Neumann algebra. The completely bounded multiplier norm of $\phi$ is defined as
$$
\|\phi\|_{M_0A} = \|M_\phi\|_{\mathrm{cb}}.
$$
The space of completely bounded multipliers is dentoed $M_0A(G)$. One has an inclusion $A(G)\subseteq M_0A(G)$. It was shown in \cite[Proposition~1.10]{MR784292} that the space $M_0A(G)$ has a predual obtained as the completion of $L^1(G)$ in the norm
$$
\|f\|_{M_0A(G)_*} = \sup\left\{\left|\int_G f(x)\phi(x)\, dx\right| : \phi\in M_0A(G),\ \|\phi\|_{M_0A}\leq 1\right\}.
$$
\begin{rem}
In general, $M_0A(G)$ does not have a unique predual nor a unique weak$^*$ topology. The predual always refers to the one constructed explicitly above, and the weak$^*$ topology on $M_0A(G)$ is the weak$^*$ topology coming from this explicit predual.
\end{rem}

Following Haagerup and Kraus \cite{MR1220905}, a locally compact group $G$ has the Approximation Property (short: AP) if there exists a net in $A(G)$ converging to $1$ in the weak$^*$ topology of $M_0A(G)$. Haagerup and Kraus showed that many groups have this property, e.g., all weakly amenable groups. This includes solvable groups, compact groups, and simple Lie groups of real rank one \cite{MR996553,MR1079871}. They showed moreover that the AP is preserved by passing to closed subgroups and preserved under group extensions (as opposed to weak amenability). It is also well-known (and routine to verify) that if $K$ is a compact normal subgroup in $G$, then $G$ has the AP if and only if $G/K$ has the AP. We refer to \cite{MR996553,MR784292,MR1220905} for details on completely bounded multipliers and the AP.

Haagerup and Kraus conjectured that the group $\mathrm{SL}(3,\mathbb R)$ does not have the AP, but the conjecture was settled only much later by Lafforgue and de~la~Salle \cite{MR2838352}. Not long after, Haagerup and de~Laat showed more generally that any simple Lie group of real rank at least two does not have the AP \cite{MR3047470,MR3453357}. Finally, a complete characterization of connected Lie groups with the AP was given in \cite{HKdL-Tstar} (see also Theorem~\ref{thm:HKdL} below). The proof in \cite{HKdL-Tstar} involved among other things the property (T$^*$), introduced in \cite{HKdL-Tstar}. We give below a much shorter proof of the characterization of connected Lie groups with the AP that avoids the use of property (T$^*$).

To state the characterization of connected Lie groups with the AP, we first recall the Levi decomposition of such a group (see \cite[Section~3.18]{MR746308} for details). For a connected Lie group $G$, one can decompose its Lie algebra $\mathfrak g$ as $\mathfrak{g}=\mathfrak{r} \rtimes \mathfrak{s}$, where $\mathfrak{r}$ is the solvable radical and $\mathfrak{s}$ is a semisimple subalgebra. One can further decompose $\mathfrak{s}=\mathfrak{s}_1 \oplus \cdots \oplus \mathfrak{s}_n$ into simple summands $\mathfrak{s}_i$ ($i=1,\ldots,n$). Let $R$, $S$, and $S_i$ denote the corresponding connected Lie subgroups of $G$. One has $G = RS$ as a set. This is called a Levi decomposition of $G$.

The subgroup $R$ is solvable, normal, and closed. The subgroup $S$ is semisimple and locally isomorphic to the direct product $S_1\times\cdots\times S_n$ of simple factors, but it need not be closed in $G$.

\begin{thm}[\cite{HKdL-Tstar}]\label{thm:HKdL}
Let $G$ be a connected Lie group, let $G=RS$ denote a Levi decomposition, and suppose that $S$ is locally isomorphic to the product $S_1 \times \cdots \times S_n$ of connected simple factors. Then the following are equivalent:
\begin{enumerate}
 \item[(i)] the group $G$ has the AP,
 \item[(ii)] the group $S$ has the AP,
 \item[(iii)] the groups $S_i$, with $i=1,\ldots,n$, have the AP,
 \item[(iv)] the real rank of the groups $S_i$, with $i=1,\ldots,n$, is at most $1$.
\end{enumerate}
\end{thm}
\begin{proof}
The equivalence of (iii) and (iv) is \cite[Theorem~5.1]{MR3453357}.

Suppose (iv) holds. We show that (i) and (ii) hold. The proof of this implication is the proof from \cite{HKdL-Tstar}. For completeness, we include it: The group $R$ is solvable, so it has the AP. Since the AP is preserved under group extenstions \cite[Theorem~1.15]{MR1220905}, it is therefore enough to prove that $G/R$ has the AP. The Lie algebra of $G/R$ is the Lie algebra of $S$, so to prove (i) and (ii), it suffices to prove that any connected Lie group $S'$ locally isomorphic to $S_1\times\cdots\times S_n$ has the AP. Using \cite[Theorem~1.15]{MR1220905} again, we may assume that the center of $S'$ is trivial. If $Z_i$ denotes the center of $S_i$, then $S' \simeq (S_1/Z_1)\times\cdots\times(S_n/Z_n)$. Each group $S_i/Z_i$ has real rank at most $1$ and hence has the AP \cite{MR996553,MR1079871}. Hence $S'$ has the AP.

Suppose (iv) does not hold. We show that (i) does not hold (which also proves that (ii) does not hold). Fix some $i$. The closure $\overline S_i$ of $S_i$ in $G$ is of the form $S_iC$ for some connected, compact, central subgroup of $G$ (see \cite[p.~614]{MR0048464}). Let $G' = G/C$, and let $\pi\colon G\to G'$ be the quotient homomorphism. As $C$ is compact, $\pi$ is a closed map, and $\pi(S_i) =\pi(\overline S_i)$ is a closed subgroup of $G'$. Since $\ker\pi\cap S_i$ is contained in the center of $S_i$ and therefore is discrete (in the topology of $S_i$), the group $\pi(S_i)$ is locally isomorphic to $S_i$. If now the real rank of $S_i$ is at least 2, then $\pi(S_i)$ does not have the AP \cite[Theorem~5.1]{MR3453357}. Hence $G'$ does not have the AP, and it follows that $G$ does not have the AP.
\end{proof}


\vfill\pagebreak
\section*{A characterization of almost connected groups with the AP}

Let $G$ be a locally compact group. Recall from \cite{HKdL-Tstar} that there is a unique left invariant mean on $M_0A(G)$, and we say that $G$ has property (T$^*$) if this mean is weak$^*$ continuous. It is clear that non-compact groups with property (T$^*$) do not have the AP. We establish the following converse for almost connected groups.

\begin{thm}\label{thm:almost-connected}
For an almost connected locally compact group $G$, the following are equivalent.
\begin{enumerate}
	\item $G$ has the AP.
	\item No closed non-compact subgroup of $G$ has property (T$^*$).
\end{enumerate}
\end{thm}
\begin{proof}
(1)$\implies$(2) is clear and holds also without the assumption of almost connectedness. We prove (2)$\implies$(1). Suppose $G$ does not have AP. We prove the existence of a closed non-compact subgroup of $G$ with property (T$^*$).

As $G$ is almost connected, 
there is a compact normal subgroup $K$ of $G$ such that $G/K$ is a Lie group (see \cite[Theorem~4.6]{MR0073104}). Then $G/K$ is a connected Lie group without the AP. It follows from the remark after Theorem~C in \cite{HKdL-Tstar} that we can find a closed, non-compact subgroup $H_0\leq G/K$ with property (T$^*$). Let $H$ be the inverse image of $H_0$ in $G$. By \cite[Proposition~5.13]{HKdL-Tstar}, $H$ has property (T$^*$), and $H$ is clearly a closed non-compact subgroup of $G$.
\end{proof}

\begin{rem}
Theorem~\ref{thm:almost-connected} is not true in general without the assumption of almost connectedness. In fact, there are discrete groups without the AP and without infinite subgroups with property (T$^*$). This is not a new result, but simply a compilation of known results, the last ingredient being the recent result of Arzhantseva and Osajda (see \cite[Theorem~2]{osajda14} and \cite{Arzhantseva-Osajda-14}) about the existence of discrete groups without property A but with the Haagerup property. Let $G$ be such a group.

As property (T$^*$) implies property (T) (see \cite[Proposition~5.3]{HKdL-Tstar}), and property (T) is an obstruction to the Haagerup property, it is clear that $G$ has no infinite subgroups with property (T$^*$).

On the other hand, if the discrete group $G$ had the AP then its reduced group C$^*$-algebra would have the strong operator approximation property. This is known to imply exactness of the group C$^*$-algebra, which implies that $G$ is an exact group. Finally, exactness for groups is equivalent to property A. This shows that $G$ cannot have the AP. Proofs of all the claimed statements can be found in \cite[Chapters~5 and 12]{MR2391387}.
\end{rem}


\begin{thebibliography}{10}

\bibitem{Arzhantseva-Osajda-14}
Goulnara Arzhantseva and Damian Osajda.
\newblock Graphical small cancellation groups with the {H}aagerup property.
\newblock Preprint, arXiv:1404.6807, 2014.

\bibitem{MR2391387}
Nathanial~P. Brown and Narutaka Ozawa.
\newblock {\em {$C^*$}-algebras and finite-dimensional approximations},
  volume~88 of {\em Graduate Studies in Mathematics}.
\newblock American Mathematical Society, Providence, RI, 2008.

\bibitem{MR996553}
Michael Cowling and Uffe Haagerup.
\newblock Completely bounded multipliers of the {F}ourier algebra of a simple
  {L}ie group of real rank one.
\newblock {\em Invent. Math.}, 96(3):507--549, 1989.

\bibitem{MR784292}
Jean de~Canni{\`e}re and Uffe Haagerup.
\newblock Multipliers of the {F}ourier algebras of some simple {L}ie groups and
  their discrete subgroups.
\newblock {\em Amer. J. Math.}, 107(2):455--500, 1985.

\bibitem{MR0228628}
Pierre Eymard.
\newblock L'alg\`ebre de {F}ourier d'un groupe localement compact.
\newblock {\em Bull. Soc. Math. France}, 92:181--236, 1964.

\bibitem{MR3047470}
Uffe Haagerup and Tim de~Laat.
\newblock Simple {L}ie groups without the {A}pproximation {P}roperty.
\newblock {\em Duke Math. J.}, 162(5):925--964, 2013.

\bibitem{MR3453357}
Uffe Haagerup and Tim de~Laat.
\newblock Simple {L}ie groups without the {A}pproximation {P}roperty {II}.
\newblock {\em Trans. Amer. Math. Soc.}, 368(6):3777--3809, 2016.

\bibitem{HKdL-Tstar}
Uffe Haagerup, S{\o}ren Knudby, and Tim de~Laat.
\newblock A complete characterization of connected {L}ie groups with the
  {A}pproximation {P}roperty.
\newblock Preprint, to appear in Ann. Sci. \'Ecole Norm. Sup., arXiv:1412.3033,
  2014.

\bibitem{MR1220905}
Uffe Haagerup and Jon Kraus.
\newblock Approximation properties for group {$C^*$}-algebras and group von
  {N}eumann algebras.
\newblock {\em Trans. Amer. Math. Soc.}, 344(2):667--699, 1994.

\bibitem{MR1079871}
Mogens~Lemvig Hansen.
\newblock Weak amenability of the universal covering group of {${\rm
  SU}(1,n)$}.
\newblock {\em Math. Ann.}, 288(3):445--472, 1990.

\bibitem{MR2838352}
Vincent Lafforgue and Mikael De~la Salle.
\newblock Noncommutative {$L^p$}-spaces without the completely bounded
  approximation property.
\newblock {\em Duke Math. J.}, 160(1):71--116, 2011.

\bibitem{MR0073104}
Deane Montgomery and Leo Zippin.
\newblock {\em Topological transformation groups}.
\newblock Interscience Publishers, New York-London, 1955.

\bibitem{MR0048464}
George~Daniel Mostow.
\newblock The extensibility of local {L}ie groups of transformations and groups
  on surfaces.
\newblock {\em Ann. of Math. (2)}, 52:606--636, 1950.

\bibitem{osajda14}
Damian Osajda.
\newblock Small cancellation labellings of some infinite graphs and
  applications.
\newblock Preprint, arXiv:1406.5015, 2014.

\bibitem{MR746308}
V.~S. Varadarajan.
\newblock {\em Lie groups, {L}ie algebras, and their representations}, volume
  102 of {\em Graduate Texts in Mathematics}.
\newblock Springer-Verlag, New York, 1984.
\newblock Reprint of the 1974 edition.

\end{thebibliography}

\end{document}